\documentclass[11pt]{article}
\usepackage{amsmath}
\usepackage{amssymb}
\usepackage{ifthen}
\usepackage{epsfig}
\usepackage{setspace}

\newcommand{\infig}[3]{
\begin{figure}
\centering
\includegraphics[width=#1in]{#2}
\caption{#3\label{fig:#2}}
\end{figure}}

\newtheorem{theorem}{Theorem}

\newtheorem{lemma}[theorem]{Lemma}

\newtheorem{proposition}[theorem]{Proposition}
\newtheorem{conjecture}[theorem]{Conjecture}

\newenvironment{proof} {\noindent{\it Proof. }} {{\qed}}
\newenvironment{proofof}[1]{\noindent{\it Proof of #1. }} {{\qed}}

\newtheorem{thm}[theorem]{Theorem}

\newtheorem{lem}[theorem]{Lemma}

\newtheorem{cor}[theorem]{Corollary}
\newtheorem{dfn}[theorem]{Definition}

\def\qed{\hspace*{\fill} $\Box$\par\medskip}

\def\mod{\mathop{\rm mod}\nolimits}

\def\<{\left<}
\def\>{\right>}

\def\Pr{\mathrm{Pr}}

\def\e{\mathrm{e}}

\title{Many Random Walks Are Faster Than One}
\author{Noga Alon\thanks{Email: nogaa@tau.ac.il} \\ Tel Aviv University
\and Chen Avin\thanks{Email: avin@cse.bgu.ac.il}  \\  Ben-Gurion University %of the Negev
\and Michal Kouck\'y\thanks{Email: koucky@math.cas.cz, work is partially supported by grant GA \v{C}R 201/07/P276 and 201/05/0124.} \\ Academy of Sciences of Czech Republic
\and Gady Kozma\thanks{Email: gady.kozma@weizmann.ac.il} \\ Weizmann Institute of Science
\and Zvi Lotker\thanks{Email: zvilo@cse.bgu.ac.il} \\  Ben-Gurion University %of the Negev
\and Mark R. Tuttle\thanks{Email: tuttle@acm.org} \\ Intel}
\date{}
\begin{document}
\maketitle
\thispagestyle{empty}

\begin{abstract}
% We initiate a systematic study of a new question regarding random walks on graphs:
We pose a new and intriguing % significant 
question motivated by distributed computing regarding random walks on graphs:
How long does it take for several independent random walks, starting from the same vertex, to cover
an entire graph?
We study the {\em cover time} -
the expected time required to visit every node in a graph at least once - and we show that for a large collection of interesting graphs, running many
random walks in parallel yields a speed-up in the cover time that is
linear in the number of parallel walks.
We demonstrate that an exponential speed-up is sometimes possible,
but that some natural graphs allow only a logarithmic speed-up.
A problem related to ours (in which the walks start from some probablistic distribution on vertices) was previously studied in the context of
space efficient algorithms for undirected $s$-$t$-connectivity %\cite{BKRU}.
and our results yield, in certain cases, an improvement upon some of the earlier bounds.

%
% MK: After reconsideration I keep the following out as the total number of steps in our consideration
% is always larger than that of a signle walk (except for bar-bell)
%
% "Our results yield in certain cases an improvement upon some of the earlier results."
%
\end{abstract}
\newpage
\pagenumbering{arabic}

\section{Introduction}

%MARK: I cut first three paragraphs, and modified next sentence.
%MARK2: I put back a modified version of the first paragraph

Consider the problem of hunting or tracking on a graph.
The prey begins at one node, the hunters begin at other nodes,
and in every step each player can traverse an edge of the graph.
The goal is for the hunters to locate and track the prey as quickly as
possible.
What is the best algorithm for the hunters to explore the graph
and find the prey?
The answer depends on many factors,
such as the nature of the graph,
whether the graph can change dynamically,
how much is known about the graph,
and how well the hunters can communicate and coordinate their actions.
Graph exploration problems such as this are particularly interesting in changing or unknown environments.
 In such environments,
randomized algorithms are at an advantage, since they typically require
no knowledge of the graph topology.

Random walks are a natural and thoroughly studied approach to
randomized graph exploration.
A \emph{simple random walk} is a stochastic process that starts at one node of
a graph,
and at each step moves from the current node to an adjacent node chosen
randomly and uniformly from the neighbors of the current node.
A natural example of a random walk in a communication network arises when
messages are sent at random from device to device.
Since such algorithms exhibit locality, simplicity, low-overhead,
and robustness to changes in the graph structure,
applications based on random walks are becoming more and more popular.
In recent years,
random walks have been proposed in the context of querying,
searching, routing, and self-stabilization in wireless ad-hoc networks,
peer-to-peer networks, and other distributed systems and applications
\cite{dolev02random,
sergio02constrained,braginsky02rumor,sadagoan03active,avin04efficient,gkants04peer,alanyali06a-random-walk,bar-yossef06rawms}.

The problem with random walks, however, is latency.
In the case of a ring, for example,
a random walk requires an expected $\Theta(n^2)$ steps to traverse a ring,
whereas a simple traversal requires only $n$ steps.
The time required by a random walk to traverse a graph, i.e., the time to {\em cover} the graph, is an important
measure of the efficiency of random walks:
The \emph{cover time} of a graph is the expected time taken by a random walk
to visit every node of the graph at least once \cite{aleliunas79random,aldos83on}.
The cover time is relevant to a wide range of algorithmic applications
\cite{gkants04peer,wagner98robotic,jerrum97markov,avin04efficient},
and methods of bounding the cover time of graphs have been thoroughly
investigated
\cite{matthews88coveringspheres,aldous89lower,chandra89electrical,broder89cover,zuckerman90lower, lovasz93survey}.
Several bounds on the cover time of particular classes of graphs have been
obtained, with many positive results
\cite{chandra89electrical,broder89cover,jonasson98on,jonasson00planar,cooper03cover}.

%In the context of solving undirected $s$-$t$-connectivity the latency translates into the running time
%of randomized algorithms. In \cite{BKRU}, Broder et al. investigate the possibility
%to speed-up the randomized space efficient algorithm for
%undirected $s$-$t$-connectivity at the expense of using multiple random walks and hence
%using proportionally more space to keep track of them. In particular, \cite{BKRU}
%studies the problem of covering a graph by multiple random walks starting from vertices chosen
%according to the stationary distribution. Their objective is to minimize the total number of steps
%taken by all the random walks together. The subsequent papers \cite{barnes-feige, feige} consider
%the same problem but using different starting distributions.
%We propose to study related but a different question of
%considerable interest.

\smallskip
The contribution of this paper is proposing and
partially answering the following question: Can multiple random
walks search a graph faster than a single random walk?
What is the cover time for a graph if we choose a node
in  the graph and run $k$ random walks simultaneously from that node,
where now the cover time is the expected time until each node has been visited
at least once by at least one random walk?

The answer is far from obvious.
Consider, for example, running $k$ random walks simultaneously on a ring.
If we start all $k$ random walks at the same node,
then the random walks have little choice but to follow each other
around the ring,
and it is simply a race to see which of them completes the trip first.
We prove in Section~\ref{s-line} that on a ring the cover time for $k$
random walks is only a factor of $\Theta(\log k)$ faster than the cover time for
a single random walk.
On the other hand, there are graphs for which $k$ random walks can
yield a surprising speed-up.
Consider a ``barbell'' consisting of two cliques of size $n$
joined by a simple path (see Figure~\ref{fig:barball} in Section~\ref{s-barbell}).
The cover time of such a graph is $\Theta(n^2)$ and its maximum is obtained  when
starting the walk from the central point of the path.
In this graph, the bells on each end of the barbell
act as a sink from which it is difficult for a single walk to escape,
%MARK: modified explanation below according to our conversation
but if a logarithmic number of random walks start at the center of the barbell,
each bell is likely to attract at least one random walk, which
will cover that part of the graph.
We prove in Section~\ref{s-barbell} that if we run $k = O(\log n)$ random
walks in parallel, starting from the center,
then the cover time decreases by a factor of $n$ from $\Theta(n^2)$ to $O(n)$,
which corresponds to a speed-up exponential in $k$.

The main result of this paper---summarized in Table~\ref{results table}---is
that, in spite of these examples,
a linear speed-up is possible for almost all interesting graphs
as long as $k$ is not too big. In Section~\ref{s-kspeed}, we prove
that if there is a large gap between the cover time and the hitting
time of a graph, where hitting time is the expected time for a
random walk to move from $u$ to $v$ for any two nodes $u$ and $v$ in
the graph, then $k$ random walks cover the graph $k$ times faster
than a single random walk
for $k$ sufficiently small (see theorems \ref{s1-thit} and \ref{c-kspeed2}).
Graphs that
fall into this class include complete graphs, expanders,
$d$-dimensional grids and hypercubes, $d$-regular balanced trees,
and several types of random graphs.
In the important special case of expanders,
we can actually prove a linear speed-up for $k \leq n$
and not just $k \leq \log n$.
While we demonstrate a relationship between the cover time and the hitting
time,
we also demonstrate a relationship between the cover time and the
mixing time (see Theorem~\ref{mixing theorem}),
which leads us to wonder whether there is some other property of a graph that
characterizes the speed-up achieved by multiple random walks more crisply
than hitting and mixing times.

%MARK: I modified opening of next paragraph since the hunter/prey
%motivation was cut from the introduction.
%MARK2: I restored the opening after restoring the hunter/prey story

% Proposed structure of the results
%
%   * table of results
%
%   * clique
%   * speed-up k  +  parallel matthews' bound
%   * mixing time
%   * barbell
%   * line
%
% Open questions:
%  * is speed-up always at most k?
%  * is speed-up always at least log k?
%  * what can you say about other classes of graphs?
%  * is there some phenomenon (say ratio between the cover time and the mixing time)
%       which determines for which ranges we get what kind of speed-up
%       (What do I mean: what is the difference between line and expander?
%         but also: for k>n^2 there is (essentially) no more speed-up on expander.)
%  * different starting points? - whole business.
%

\begin{table}%[htdp]
\caption{Results summary (for any constant $\epsilon > 0$)}
\small
\label{results table}
\centering
\begin{tabular}{|l|c|c|c|c|c|}
\hline
Graph family name  & $\text{Cover time} $ & $\text{Hitting time} $ & $\text{Mixing time}$ &\multicolumn{2}{c|}{Speed up $S_k$ (order)} \\
                                & $C $ & $h_{max}$ & $t_m$ & lower bound & upper bound  \\ \hline \hline
cycle& $n^2/2$ & $n^2/2$ & $O(n^2)$ & $\log(k)$ & $\log(k)$ \\ \hline
2-dimensional grid & $\Theta(n\log^2 n)$& $\Theta(n\log n)$ &$\Theta (n)$ & $k$,  $k<O(\log^{1-\epsilon} n)$ & \\ \hline
d-dimensional grid, $d>2$ & $\Theta(n \log n)$ & $ \Theta(n)$ & $\Theta(n^{2/d})$ & $k$,  $k<O(\log^{1-\epsilon} n)$  &  \\ \hline
hypercube & $\Theta(n \log n )$ &$\Theta (n)$ & $\log n \log\log n$ & $k$,  $k<O(\log^{1-\epsilon} n)$  & \\ \hline
complete graph &$\Theta(n \log n) $& $\Theta (n)$ & $1$ & $k$, $k<n$ &$k$, $k<n$  \\\hline
expanders & $\Theta(n \log n )$ & $\Theta (n)$ & $\log n$ & $\Omega(k)$,  $k<n$  &  \\ \hline
E-R Random graph & $\Theta(n \log n)$ & $\Theta (n)$ & $\log n$ & $k$,  $k<O(\log^{1-\epsilon} n)$  & \\\hline
% barbell & $\Theta(n^2)$ & $\Theta(n^2)$ & $\Theta(n^2)$ & open ? & $n$, $k=20 \ln n$ \\\hline
\end{tabular}
\label{tbl:results}
\end{table}%

\subsection{Related work}

A related problem was previously studied in the context of algorithms for solving undirected $s$-$t$ connectivity,
the problem of deciding whether two given vertices $s$ and $t$ are connected in an undirected graph.
%Chen:Shorten due to space
%The following generic scheme is used in multiple algorithms: In order to decide whether two given vertices $s$ and $t$ are
%connected in an undirected graph $G$, we (possibly repeatedly) reduce the original question on the graph $G$ to the same
%type of question on a smaller graph $G'$. The graph $G'$ is in general obtained from $G$ by identifying large subsets of
%connected vertices in $G$ and collapsing all the vertices in each subset to a single representative. Depending on the
%algorithm we insert certain edges between the representatives in $G'$.
The key step in many of these algorithms is to identify large subsets of connected vertices and to shrink the graph accordingly. 
The algorithms use short random or pseudorandom walks to identify such subsets. These walks are
either starting from all the vertices of $G$ or from a suitably chosen sample of its vertices. Deterministic algorithms
concerned with the amount of used space \cite{NSW,ATWZ} use pseudorandom walks started from all the vertices of $G$. Parallel randomized
algorithms, e.g., \cite{KNP, HZ}, use short random walks from each vertex of $G$. Although there seems to be a deeper connection to our problem, these techniques
do not seem to provide any results directly related to our question of interest.

However, a problem closer to ours is considered in a sequence of papers on time-space trade-offs for solving $s$-$t$-connectivity \cite{BKRU, barnes-feige, feige}. Algorithms in this area choose first a random set of
representatives and then perform short random walks to discover connectivity between the representatives. 
A part of the analysis in \cite{BKRU} by Broder \emph{et al.} is calculating the expected number of steps needed to cover the whole graph. 
Indeed, Broder \emph{et al.} state as one of their main results that the
expected number of steps taken by $k$ random walks starting  from $k$ vertices chosen according to the stationary
distribution to cover the whole graph is $O(\frac{m^2\log^3 n}{k^2})$, where $m$ is the number of edges and $n$ is
the number of vertices of the graph \cite{BKRU}. Barnes and Feige in \cite{barnes-feige, feige} consider different starting distributions that give a
better time-space trade-off for the $s$-$t$-connectivity algorithm but they do not state any explicit bound on the cover time by $k$ random walks.
In contrast, in this work, we formulate our interest in comparison between the expected cover time of a single walk and of $k$ random walks.

Although our work focuses on covering the graph starting from a single vertex, under certain conditions 
%for graphs with fast mixing time or cover time of $o(mn/\log n) 
our results
yield improved bounds on the cover time starting from the stationary distribution. In particular, for graphs with fast mixing time, Lemma~\ref{l91} yields
the bound $O((n \log n)/k)$ on the cover time of $k$ random walks starting from the stationary distribution on an expander and
the proof of Theorem \ref{mixing theorem} gives bound $O((n t_m \log^2 n )/ k)$ on the cover time of $k$  random walks starting
from the stationary distribution on graphs with mixing time $t_m$. Indeed, our proofs in Section \ref{s-kspeed} do not depend on the starting
distribution so similar results can be stated for k walks starting from an arbitrary probabilistic distribution. 

% MK: barnes-feige and feige seem to modify the graph on which they are making % the walk by adding self-loops which should not affect the walk.
% Their work seem to correspond (numerically) to cover time of O(mn/p^2).

%%=========================== Preliminaries ===================================================

%\section{Related Work}
% theory results on random walk.

\section{Preliminaries}
% graph notation
% random walk
% cover time
% hitting time
% Matthew's bound ?
% Time to reach half of the nodes (2h_\max}
% stopping time?
Let us begin with a quick review of asymptotic notation, like $o(1)$,
as used in this paper:
$f(n) = O(g(n))$ if there exist positive numbers $c$ and $N$,
such that $f(n) \le c g(n), \forall n\ge N$.
$f(n) = \Omega(g(n))$ if there exist positive numbers $c$ and $N$,
such that $f(n) \ge c g(n), \forall  n \ge N$.
$f(n) = \Theta(g(n))$ if $f(n)=O(g(n))$ and $f(n)=\Omega(g(n))$.
$f(n) = o(g(n))$ if $\lim_{n \to \infty} f(n)/g(n) = 0$ and
$f(n) = \omega(g(n))$ if $\lim_{n \to \infty} f(n)/g(n) = \infty$.

Let $G(V, E)$ be an undirected graph, with $V$ the set of nodes and $E$ the set
of edges.
Let $n = \lvert V \rvert$ and $m = \lvert E \rvert$.
For $v \in V$,
let $N(v) = \{u \in V \mid (v,u) \in E\}$ be the set of neighbors of $v$,
and let $\delta(v) = \lvert N(v) \rvert$ be the degree of $v$.
A $\delta$-regular graph is a graph in which every node has degree $\delta$.

Let $X_i = \{X_i(t) : t \ge 0\}$ be a \emph{simple random walk} starting from node $i$ on the state space $V$ with \emph{transition matrix} $Q$.
When the walk is at node $v$, the probability to move in the next step to $u$ is $Q_{vu}= \Pr(v,u) = \frac{1}{\delta(v)}$ for $(v,u) \in E$ and $0$ otherwise.
%The walk starts at node $i$ and the  node at time $t+1$ is chosen uniformly at random from the set of neighbors of $N(X_i(t))$,

Let $\tau_i(G)$ of a graph $G$ be the time taken
by a simple random walk starting at $i$ to visit all nodes in $G$. Formally $\tau_i = \min \{t : \{X_i(1), \dots, X_i(t)\}=V\}$ and clearly this is a stopping time and therefore a random variable.
Let $C_i = E[\tau_i]$ be the expected number of steps for the
simple random walk starting at $i$ to visit all the nodes in $G$.
The {\em cover time} $C(G)$ of a graph $G$ is defined formally as $C(G) = \max_i C_i$.
%Similarly we can define the variance of the cover time (with an abuse of notation) as $\sigma^2(G) = \max_i var(\tau_i)$.
The {\em cover time} of graphs and methods of bounding it have been extensively
investigated \cite{matthews88coveringspheres, aldous89lower,
chandra89electrical, broder89cover, zuckerman90lower,
aleliunas79random}, although much less is known about the variance of the cover time.  Results for the cover time of specific graphs
vary from the {\em optimal cover time} of $\Theta(n \log n)$
associated with the complete graph $K_n$ to the worst case of
$\Theta(n^3)$ associated with the lollipop graph \cite{feige95upper,
feige95lower}.

The {\em hitting time}, $h(u, v)$, is the expected time for a random walk starting at $u$ to arrive to $v$ for the
first time. %and the {\em commute time},  $C_{uv}$, is the expected time for a random walk starting at $u$ to first
%arrive at $v$ and then return to $u$.
Let $h_{\max}$  be the maximum $h(u,v)$ over all ordered pairs of nodes and let $h_{\min}$ to be defined similarly. The following theorem
provides fundamental bound on the cover time $C(G)$ in terms of $h_{\max}$ and $h_{min}$.
\begin{thm}[Matthews' Theorem \cite{matthews88coveringspheres}]\label{thm:matthew} For any  graph G,
\begin{displaymath} h_{min}\cdot H_n \:\: \le \:\: C(G) \:\: \le \:\: h_{\max} \cdot H_n
\end{displaymath}
\noindent where $H_k = \ln(k) + \Theta(1)$ is the k-th harmonic number.
\end{thm}
Notice that this bound  is not always tight, since in the line, for example, we have  $C(G) = h_{\max}$.

For an integer $t>0$, a graph $G$ and its vertices $u$ and $v$, let $p^t_{u,v}$ be the probability that a simple random
walk starting from vertex $u$ is at vertex $v$ at time $t$ and let $\pi(v)$ denote the probability of being at $v$ under the stationary distribution of $G$.
By {\em mixing time $t_m$ of $G$}, we understand the smallest integer $t>0$ such that for all vertices $u$ in $G$, $\sum_v |p^t_{u,v}-\pi(v)| < 1/\e$.

%
%In \cite{lovasz93survey} the followoing lemma was proven:
%\begin{lem}
%For any  graph G let $C_{G}(\frac{1}{2})$ be the expected number of steps before a random walk
%visits more than half of the nodes, then:
%\begin{displaymath} C_{G}(\frac{1}{2}) \:\: \le \:\: 2h_{\max}
%\end{displaymath}
%\end{lem}

\subsection{$k$-Random Walks: Cover Time and Speed-up}
Let us turn our attention to the case of $k$ parallel independent random walks. We assume that all walks start from the same node and
we are interested in the performance of  such a system.
The natural extension to the definition of cover time is the  $k$ cover time:
Let $\tau_i^k$ be the random time taken
by $k$ simple random walks, all starting at $i$ at $t=0$, to visit all nodes in $G$ (i.e., the time by which each node has been visited by at least one of the walks).
Let  $C^k_i = E[\tau_i^k]$ be the expected cover time for $k$ walks starting from $i$.
For a graph $G$, let $C^k(G) = \max_i C^k_i(G)$ be the $k$-walks' cover time. In practice, we would like to bound the speed-up in the expected cover time achieved by $k$ walks:
\begin{dfn}
For a graph $G$ and an integer $k>1$, the {\em speed-up}, $S^{k}(G)$, on $G$, is the ratio between
the cover time of a single random walk and the cover time of $k$ random walks, namely, $S^{k}(G) = \frac{C(G)}{C^{k}(G)}$.
\end{dfn}
Note that speed-up on a graph is a function of $k$ and of the graph. When $k$ and/or graph is understood from the context we may not mention them explicitly.

\section{Statement of our results}

We show that $k$ random walks can cover a graph $k$ times faster than 
a single random walk on a large class of graphs, 
%MARK3: modified next line
a class that includes many important and practical instances.
%MARK3: added next sentence
We begin with a simple statement of linear speed-up on simple graphs, 
but as we broaden the class of graphs considered,
our statements of speed-up become more involved.
We begin with a linear speed-up on cliques and expanders:

\begin{theorem}
For $k\le n$ and for a graph $G$ that is either a complete graph on $n$ vertices or an expander the speed-up is $S^k(G) = \Omega(k)$.
\end{theorem}

\noindent
We can show a linear speed-up on other graphs, as well, but to do so we must bound $k$, the number of random walks.
Which bound we use depends on Matthews' bound.

When Matthews' bound is tight, we can prove a linear speed-up for $k$ as large as $k \leq \log n$.
Our proof depends on a generalization of Matthews' bound for multiple random walks:
$C^k(G) \le {\e + o(1) \over k} \cdot h_{\max} \cdot H_n$ (see Theorem \ref{thm:generalMatthew}).
Since Matthews' bound is known to be tight for the complete graph,
expanders \cite{chandra89electrical},
$d$-dimensional grids for $d \ge 2$ \cite{chandra89electrical},
$d$-regular balanced trees for $d \ge 2$ \cite{Zuckerman:1989lr},
Erd\H{o}s-R\'enyi random graphs \cite{cooper03cover},
and random geometric graphs \cite{avin05on-the-cover} (in the last two cases,
 for choice of parameters that guarantee connectivity 
with high probability), 
the following result shows that $k \leq \log n$ random walks yield a linear 
speed-up for a large class of interesting and useful graphs:

\begin{thm}\label{s1-thit}
If $C(G) = \Theta(h_{\max} \log n)$,
then $S^k(G) = \Omega(k)$ for all $k \leq \log n$.
\end{thm}

%MARK3: ``must proceed more indirectly''
When Matthews' bound is not tight, we must proceed more indirectly and bound $k$ in terms of \emph{the gap}.
Let $g(n) = \frac{C}{h_{\max}}$ be the gap between the cover time and the maximum hitting time.
We find it remarkable that, 
using this gap, we can prove a nearly linear speed-up for $k$ less than $g(n)$ without 
knowing the actual cover time.

\begin{thm}\label{c-kspeed2}
If $g(n)= \frac{C(G)}{h_{\max}} \to \infty$
and $k \le O(g^{1-\epsilon}(n))$ for some $\epsilon < 1$,
then $C^{k}(G) = \frac{C(G)}{k} + o(\frac{C(G)}{k})$, and $S^{k}(G) \ge k - o(k)$.
\end{thm}

These results raise several interesting questions about speed-ups on graphs in general:
is $k$ an upper bound on the best possible speed-up,
does proving a linear speed-up generally require bounding $k$,
%MARK3: added next line
and what really characterizes the best possible speed-up?

For the first question,
we have been unable to prove that $k$ is an upper bound on the best possible speed-up.
%MARK3 modified next sentence
We do know that a wide range of speed-ups is possible, 
and that sometimes the speed-up can be much less than $k$.
The following result shows that the speed-up on a cycle is limited to $\log k$.

\begin{theorem}\label{thm:cycle}
For all $k<\e^{n/4}$, the speed-up on the cycle $L_n$ with $n$ vertices
is $S^k(L_n) = \Theta(\log k)$.
%For $k\le 2^n$ and for a cycle $L_n$ on $n$ vertices the speed-up is $S^k(L_n) = \Theta(\log k)$.
\end{theorem}

\noindent
On the other hand, it is possible there are graphs for which the speed-up is much more than $k$.
For example, the following result shows that, when the walk starts at the node in the center of 
the graph 
%MARK3: modified next line for clarity
(we can't prove this is true from other nodes in the graph), 
the speed-up is \emph{exponential} in $k$:

\begin{theorem}\label{thm:bar-bell}
For a bar-bell graph $B_n$ on $n$ vertices (see Section~\ref{s-barbell} for a definition) if $v_{\rm c}$ is
the center of the bar-bell then $C_{v_{\rm c}} = \Theta(n^2)$ but $C_{v_{\rm c}}^k = O(n)$ for $k=\Theta(\log n)$.
\end{theorem}

%MARK3: added ``situation turns out to be rather complex''
For the second question, proving a linear speed-up in general does indeed require bounding $k$.
In fact, the situation turns out to be rather complex, 
since the speed-up depends not only on the graph itself,
but also on the relationship between the size of the graph and $k$. 
For example, using Theorem~\ref{thm:cycle}, we can show that there may be a full spectrum of 
speed-up behaviors even for a single graph:

\begin{theorem}\label{t-grid}
Let $G$ be a two dimensional $\sqrt{n} \times \sqrt{n}$ grid on the torus 
%MARK3: modified next line for page break reasons
(for which Matthews' bound is tight).
\begin{enumerate}
\item For $k\le \log n$, the speed-up is $S^k(G) = \Omega(k)$
\item For $k\ge \log^3 n$ the speed-up is $S^k(G) = o(k)$.
\end{enumerate}
\end{theorem}

Finally, what property of a graph determines the speed-up?  We do not have a complete answer to this question.
We are able to relate the speed-up on a graph to the ratio between cover time of the graph and the maximal hitting
time of the graph as seen in Theorem~\ref{c-kspeed2} and further also to the mixing time of the graph. Intuitively
if a graph has a fast mixing time then the random walks spread in different parts of the graph and explore it essentially
independently.

\begin{thm}
\label{mixing theorem}
Let $G$ be a $d$-regular graph. If the mixing time of $G$ is $t_m$ then for $k \le n$ the speed-up is $S^k = \Omega(\frac{k}{t_m \ln n})$
\end{thm}

\noindent
To this end, questions regarding minimal and maximal bounds on the speed-up as a function of $k$ remain open, but we do conjecture that speed-up
is at most linear and at least of logarithmic order:

\begin{conjecture}
For any graph $G$ and any $k\ge 1$, $S^k(G)\le O(k)$.
\end{conjecture}

\begin{conjecture}
For any graph $G$ and any $n \ge k\ge 1$, $S^k(G)\ge \Omega(\log k)$.
\end{conjecture}

\section{Linear speed-up}\label{s-kspeed}

Linear speed-up in a clique follows from folklore, and we will show
linear speed-up in an expander in Section~\ref{sec:expanders}, so we
begin by stating this simple example from folklore for later use:

\begin{lem}\label{lem:clique}
For $k\le n$ and a clique $K_n$ of size $n$ the speed-up is $S^k(K_n) = k$ (up-to a rounding error).
\end{lem}

\noindent
We now show a linear speed-up in a much larger class of graphs,
as long as $k \leq \log n$.
We begin with Matthews' upper bound $C(G) \leq h_{\max} \cdot H_n$
for the cover time by a single random walk,
and generalize the bound to show that $k$ random walks improve Matthews' bound by a linear factor:

\begin{thm}[Baby Matthew Theorem]\label{thm:generalMatthew}
If $G$ is a graph on $n$ vertices and $k \leq \log n$, then
$$C^k(G) \le {\e + o(1) \over k} \cdot h_{\max} \cdot H_n.$$
\end{thm}

\begin{proof}
Let the starting vertex $u$ of the $k$-walk be chosen. Fix any other vertex $v$ in the graph $G$.
Recall, for any two vertices $u',v'$ in $G$, $h(u',v')\le h_{\max}$. Thus by Markov inequality,
$\Pr[$a random walk of length $\e h_{\max}$ starting from $u$ does not hit $v ] \le 1/\e$.
Hence for any integer $r>1$, the probability that a random walk of length $\e r h_{\max}$
does not visit $v$ is at most $1/\e^r$. (We can view the walk as $r$ independent trials
to visit $v$.) Thus the probability that a random $k$-walk of length $\e r h_{\max}$ starting from $u$ does
not visit $v$ is at most $1/\e^{kr}$. Set $r=\lceil (\ln n + 2 \ln \ln n)/k \rceil$. Then the
probability that a random $k$-walk of length $\e r h_{\max}$ does not visit $v$
is at most $1/(n \ln^2 n)$. Thus with probability at least $1-(1/\ln^2 n)$ a random $k$-walk
visits all vertices of $G$ starting from $u$. Together with Matthews' bound $C(G)\le h_{\max} H_n$,
we can bound the $k$-cover time of $G$ by $C^k(G) \le \e r h_{\max} (1+1/\ln^2) + C(G)/\ln^2 n \le (\e+o(1)) h_{\max} H_n/ k$.
The theorem follows.
\end{proof}

When Matthews' bound is tight,
we have $C(G) = \Theta(h_{\max} \log n)$,
and the linear speed-up is an immediate corollary of
Theorem~\ref{thm:generalMatthew}:
%\begin{corollary}\label{c-mat-log}
\\ \\ %\vspace{0.2in}
\textbf{Theorem~\ref{s1-thit}}
\emph{
If $C(G) = \Theta(h_{\max} \log n)$,
then $S^k(G) = \Omega(k)$ for all $k \leq \log n$.}
%\end{corollary}
\\ \\ %\vspace{0.2in}
%Since Matthews' bound is known to be tight for
%the complete graph,
%expanders \cite{chandra89electrical},
%$d$-dimensional grids for $d \ge 2$ \cite{chandra89electrical},
%$d$-regular balanced trees for $d \ge 2$ \cite{Zuckerman:1989lr},
%Erd\H{o}s-R\'enyi random graphs \cite{cooper03cover},
%and random geometric graphs \cite{avin05on-the-cover},
%we see that $k \leq \log n$ random walks yield a linear speed-up for a
%large class of interesting and useful graphs.
When Matthews' bound is not tight, the proofs become more complex.
We begin with the following result expressing the $k$-walk cover time
in terms of the single-walk cover and hitting times:

\begin{thm}\label{thm:speedup}\label{t-kspeed}
For any graph G of size $n$ large enough and for any function $f(n) \in \omega(1)$
\begin{displaymath}
C^{k}(G) \:\: \le \:\: \frac{(1+o(1))}{k}\cdot C(G) + (3\log k + 2f(n))\cdot h_{\max}.
\end{displaymath}
\end{thm}

\noindent
The proof is at the end of the section. In this case,
we get at least an order of linear speed-up when this upper bound is dominated by the left term.
%MARK: I added initial weasel words, and also removed the big-oh notation
%so it would be clear that we are making no pretense that this is a real
%analysis.
Choosing $f(n)$ sufficiently small,
informal calculation shows
this happens when $\log k \cdot h_{\max} \leq C/k$
or $k \log k \leq C/h_{\max}$,
which happens when $k = (C/h_{\max})^{1-\epsilon}$.
Once again, when Matthews' bound is tight and $C/h_{\max} = \log n$ we
have the following approximation to our previous result,
%MARK: I added this comment about improving the constant to 1
which improves the linear speed-up constant from $1/\e$ to $1$ at the
cost of a slight reduction in the choice of applicable $k$:

\begin{cor}\label{c-kspeed}
If $C=\Theta(h_{\max}\log n)$ and $k = O({\log^{1-\epsilon} n})$ for
some $\epsilon < 1$,
then $C^{k} = \frac{C}{k} + o(\frac{C}{k})$, and $S^{k}(G) \ge k - o(k)$.
\end{cor}

\noindent
When Matthews' bound is not tight,
we have the following result expressed directly in terms of the gap
$g(n) = \frac{C}{h_{\max}}$ between the cover time and the hitting time:
%\begin{cor}\label{c-kspeed2}
\\ \\ %\vspace{0.2in}
\noindent \textbf{Theorem~\ref{c-kspeed2}}
\emph{
If $g(n)= \frac{C(G)}{h_{\max}} \to \infty$
and $k = O(g^{1-\epsilon}(n))$ for some $\epsilon < 1$,
then $C^{k}(G) = \frac{C}{k} + o(\frac{C}{k})$, and $S^{k}(G) \ge k - o(k)$.}
%\end{cor}
\\ \\ %\vspace{0.2in}
\begin{proof}
Set $f(n) \in \omega(1)$ in Theorem~\ref{thm:speedup} to be $\log(g(n))$,
and the claim follows.
\end{proof}

%We repeat Theorem~\ref{thm:aldous} as one can obtain our main results for $k$ constant using it, instead of Theorem~\ref{thm:var}.

%\begin{cor}
%For the simple random walk on $G$, starting at $i$, if $C_i/h_{\max} \to \infty$ then
%\begin{displaymath}
%\exists f \in o(C_i) \quad s.t. \quad \Pr[\tau_i \le C_i+f] \rightarrow 1
%\end{displaymath}
%or in words the time to reach cover is, with high probability, at most $C_i + o(C_i)$.
%\end{cor}
We now prove Theorem~\ref{t-kspeed}.
Our main technical tool conceptually different from our previous proofs is the following lemma
(see the appendix for the proof).

\begin{lemma}\label{l-composition}
Let $G$ be a graph and $u_1,\dots,u_k$ be some of its vertices, not necessarily distinct.
Let $T_c$ and $p_c$ be such that a random walk of length $T_c$ starting from $u_1$ visits
all vertices of $G$ with probability at least $p_c$. Let $T_h$ and $p_h$ be such
that for any two vertices $u$ and $v$ of $G$, a random walk of length $T_h$ starting
from $u$ visits $v$ with probability at least $p_h$. Let $\ell>1$ be an integer.
Then a random $k$-walk of length $T_c/k + \ell T_h$ starting from vertices $u_1,\dots,u_k$
covers $G$ with probability at least $p_c (1-k(1-p_h)^\ell)$.
\end{lemma}

Next, we use the following bound on the concentration of the cover time
by Aldous \cite{aldous91threshold}:
\begin{thm}[\cite{aldous91threshold}]\label{thm:aldous}
For the simple random walk on $G$, starting at $i$, if $C_i/h_{\max} \to \infty$ then \\ $\tau_i/C_i \xrightarrow{p} 1$.
\end{thm}

Equipped with the proper tools we are ready to prove Theorem~\ref{t-kspeed}.

\begin{proofof}{Theorem~\ref{t-kspeed}}
If the conditions of Theorem~\ref{thm:aldous} do not hold then the cover time and hitting time are on the same order and Theorem~\ref{t-kspeed} gives a trivial
(not tight) upper bound. Assume the conditions of Theorem~\ref{thm:aldous} holds. Theorem~\ref{thm:aldous} implies that $\Pr[\tau_u/C_u > 1+\delta_n] \le \epsilon_n$
where $\delta_n,\epsilon_n \rightarrow 0$ as the size of the graph goes to infinity.
Thus $\Pr[$a random walk of length $(1+o(1))\cdot C$ covers $G]\ge 1-o(1)$.
By Markov bound, for a fixed vertex $v$ of the graph, $\Pr[$a random walk
of length $2h_{\max}$ visits vertex $v]\ge 1/2$. If we set $\ell=\log k + \omega(1)$,
then Lemma~\ref{l-composition} implies that a random $k$-walk of length
$L={(1+o(1)) C \over k}  + (\log k +\omega(1))2h_{\max}$ covers $G$ with probability at least
$(1-o(1))\cdot (1-k2^{-\ell}) = (1-o(1))\cdot\left(1-{1\over \omega(1)}\right) =  1-o(1)$.
Here each of the $k$ random walks may start at a different vertex.
%If at each round we issue $k$-random walk of length $L$, the probability of success is $1-o(1)$ and so the expected number
% of rounds is $\frac{1}{1 - o(1)}=1+o(1)$ and the claim follows.
Thus a walk of length $i\cdot L$ does not cover $G$ with probability at most $[o(1)]^i$ so the cover time of $G$ can be bounded by 
$L \sum_i i \cdot [o(1)]^i = L \cdot \frac{1}{1 - o(1)}= L \cdot (1+o(1))$.
%The claim follows.
\end{proofof}

%It is not immediately clear whether the limitation on $k$ from Corollaries~\ref{s1-thit} and~\ref{c-kspeed}
%is real or whether it is just an artifact of our proofs. The following claim implies that it cannot be only artifact of our
%proofs.
%
%\begin{thm}
%Let $G_{n,2}$ be a $2$-dimensional grid (torus) on $\sqrt{n}\times \sqrt{n}$ vertices.
%For any $k$, $S^k(G_{n,2}) \le O(\log^2 n \log k)$.
%\end{thm}
%
%So in particular for $k\approx n^\epsilon$, the speed-up is at most poly-logarithmic on a 2-dimensional grid. We defer proof of this
%theorem to Section~\ref{s-line}. In the next section we prove though that for expanders, the speed-up is at least linear
%for $k$ up-to $n$.

%%============================== Expanders =======================================================
%MARK: modified subsection title to be consistent with ``linear speed-up''
% section title
\subsection{Linear speed-up on expanders}\label{sec:expanders}
%MARK: added a transition paragraph
In this section we prove that for the important
special case of expanders,
there is a linear speed-up for $k$ as large as $k \leq n$:

\begin{thm}\label{thm:expanders}
If $G$ is an expander, then the speed-up $S^k(G)=\Omega(k)$ for $k \le n$.
\end{thm}

An $(n,d, \lambda)$-graph is a $d$-regular graph $G$ on $n$ vertices
so that the absolute value of every nontrivial eigenvalue of the adjacency
matrix of $G$ is at most $\lambda$.
It is well known (see \cite{alon86eigen}) that a  $d$-regular graph
on $n$ vertices
(with a loop in every vertex) is an expander (that is, any set $X$ of at most
half the
vertices has at least $c|X|$ neighbors outside the set, where  $c>0$ is bounded away
from zero),
if and only if there is a fixed $\lambda$ bounded away from $d$ so that
$G$ is an $(n,d,\lambda)$-graph. Since the rate of convergence of a random walk to a
uniform distribution is determined by the spectral properties of the graph it will
be convenient to use this equivalence and prove that random walks on
$(n,d,\lambda)$-graphs,  where $\lambda$ is bounded away from $d$, achieve linear speed
up. In what follows we make no attempt to optimize the absolute constants,
and omit all floor and ceiling signs whenever these are not crucial.
All logarithms are in the natural basis $e$ unless otherwise specified.

\begin{lemma}
\label{l91}
Let $G$ be an $(n,d,\lambda)$-graph. Put $s=\frac{\log (2n)}{\log (d/\lambda)}$
and $b=\frac{\lambda}{d-\lambda}$.
Then, for every two vertices $u,v$ of $G$, the probability that a random walk of length
$2s$ starting
at $u$, covers $v$ is at least $\frac{s}{2n+4s+4bn}$.
\end{lemma}
\begin{proof}
For each $i$, $s<i \leq 2s$, let $Y_i$ be the indicator random variable whose value
is $1$ iff the walk starting at $u$ visits $v$ at step number $i$. Let
$Y=\sum_{i=s+1}^{2s} Y_i$ be the number of times the walk visits $v$ during its last $s$
steps. Our objective is to show that the probability that $Y$ is positive is at least
$\frac{s}{2n+4s+4bn}$. To do so, we estimate the expectation of $Y$ and of $Y^2$
and use the fact that by Cauchy-Schwartz
\begin{equation}
\label{e91}
\Pr[Y>0]=\sum_{j>0} \Pr[Y=j] \geq
\frac{(\sum_{j>0} j \Pr[Y=j])^2}{ \sum_{j>0} j^2 \Pr[Y=j]}
=\frac{(E(Y))^2}{E(Y^2)}
\end{equation}

By linearity of expectation $E(Y)=\sum_{i=s+1}^{2s} E(Y_i)$. The expectation of $Y_i$
is the probability the walk visits $v$ at step $i$. This is precisely the value of the
coordinate corresponding to $v$ in the vector $A^i z$, where $A$ is the stochastic
matrix of the random walk, that is the adjacency matrix of $G$ divided by $d$,
and $z$ is the vector with $1$ in the coordinate $u$ and $0$ in every other coordinate.
Writing $z$ as a sum of the constant $1/n$-vector $z_1$ and a vector $z_2$
whose sum of coordinates
is $0$, and using the fact that $Az_1=z_1$ and that the $\ell_2$-norm of $A^iz_2$
satisfies $|| A^i z_2 ||  \leq (\frac{\lambda}{d})^i$ we conclude, by the
definition of $s$, that each coordinate of $A^iz$ deviates from $1/n$ by at most
$\frac{1}{2n}$.

It thus follows that
\begin{equation}
\label{e92}
E(Y) \geq \frac{s}{2n}.
\end{equation}
By linearity of expectation
$$
E(Y^2)=\sum_{i=s+1}^{2s} E(Y_i) + 2\sum_{s <i <j \leq 2s} E(Y_i Y_j)
$$
Note that $E(Y_iY_j)$ is precisely the probability that the walk visits $v$ at step
$i$ and  at step $j$. This is the probability that it visits $v$ at step $i$, times the
conditional probability that it visits $v$ at step  $j$ given that it visits it at step
$i$. This conditional probability can be estimated as before, showing that it deviates
from $1/n$ by at most $(\lambda/d)^{j-i}$. It thus follows that
\begin{equation}
\label{e93}
E(Y^2) \leq E(Y) +2 \sum_{i=s+1}^{2s} E(Y_i)(\frac{s}{n}+\sum_{r>0} (\lambda/d)^r)
\leq E(Y) [1+\frac{2s}{n}+2\frac{\lambda}{d-\lambda}].
\end{equation}
Plugging the estimates (\ref{e92}) and (\ref{e93}) in (\ref{e91}) we conclude that
$$
\Pr[Y>0] \geq \frac{(E(Y))^2}{E(Y)[1+2s/n+2\lambda/(d-\lambda)]}
\geq \frac{s/(2n)}{1+2s/n+2b}=\frac{s}{2n+4s+4bn}.
$$
This completes the proof.
\end{proof}
\begin{cor}
\label{c92}
Let $G$ be an $(n,d,\lambda)$-graph and define $s=\frac{\log (2n)}{\log(d/\lambda)}$,
$b=\frac{\lambda}{d-\lambda}$. Suppose $n \geq 2s$, and let $k$ be an integer so
that $\frac{16(b+1)n \log n}{k}>2s$. For any two fixed vertices $u$ and $v$ of
$G$, the probability that $v$ is not covered by at least one of $k$ independent random
walks starting at $u$, each of length $t=\frac{16(b+1)n \log n}{k}$, is smaller than
$\frac{1}{n^2}$.
\end{cor}
\begin{proof}
Break each of the walks into $\frac{t}{2s} $ sub-walks,
each of length $2s$. By Lemma~\ref{l91}, for each of these sub-walks, the probability
it covers $v$ is at least $\frac{s}{2n+4s+4bn}\geq \frac{s}{4(b+1)n}.$ Note that
this estimate holds for each specific sub-walk, even after we expose all previous
sub-walks, as given this information it is still a random walk of length $2s$ starting
at some vertex of $G$, and this initial vertex is known once the previous sub-walks are
exposed. It follows that the probability that $v$ is not covered
is at most
$$
(1-\frac{s}{4(b+1)n})^{kt/2s} < e^{-kt/(8(b+1)n}=e^{-2 \log n}=\frac{1}{n^2},
$$
as needed.
\end{proof}
In the notation of the above corollary, the $k$ random walks of length $t$
starting at $u$ cover the
whole expander with probability at least $1-1/n$. Since the usual cover time of the
expander is $O( n \log n)$ it follows that the expected length of the walks
until they cover the graph does not exceed $t+\frac{1}{n} O( n \log n) \leq O(t)$.

Note that for every fixed $b$,
the total length of all $k$ walks in the last corollary is $O( n \log n)$, and that the
assumption $\frac{16(b+1)n \log n}{k}>2s=2\frac{\log (2n)}{\log(d/\lambda)}$ holds
for every $k$ which does not exceed $b'n$ for some absolute constant $b'$ depending
only on $b$ (as $d/\lambda=1+1/b$).
This shows that $k$ random walks on
$n$-vertex expanders achieve speed-up $\Omega(k)$ for all $k \leq n$.

%%============================== Mixing time =======================================================
\section{Speed-up and Mixing Time}
%MARK: modified the next transition paragraph.
Random walks on expanders converge rapidly to the stationary distribution.
For graphs with fast mixing times, like expanders,
the following theorem gives a second bound on the
speed-up in terms of mixing time.
%\begin{thm}
\\ \\ %\vspace{0.2in}
\noindent \textbf{Theorem~\ref{mixing theorem}}
\emph{
Let $G$ be a $d$-regular graph. If the mixing time of $G$ is $t_m$ then for $k \le n$ the speed-up is $S^k = \Omega(\frac{k}{t_m \ln n})$}
%\end{thm}
\\ \\ %\vspace{0.2in}
\begin{proof}
Let $G$ be a $d$-regular graph of size $n$. We show that the expected
cover time of $G$ by a random $k$-walk is $O({ t_m n \ln^2 n\over k})$.
As a cover time of any graph is at least $n\ln n$ the theorem follows.

In this proof we represent a random $k$-walk on $G$ by an infinite
sequence of random variables $X_0,X_1,\dots$, where $X_i$ is the position
of the $1+(i \mod k)$-th token at step $\lfloor i/k \rfloor +1$. Define
the random variables $Y_i=X_{\lfloor i/k \rfloor k \cdot 6 t_m \ln n +
(i\mod k)}$. Hence, $Y_i$'s correspond to the position of the $k$-walk
after every $6 t_m \ln n$ steps. Let a random variable $Y'_i$ be $Y_i$
conditioned on a specific outcome of $Y_0,\dots,Y_{i-k}$. Since $t_m$ is
the mixing time of $G$ and the stationary distribution of a random walk on
$G$ is uniform ($G$ is $d$-regular), the statistical distance of $Y'_i$
from the uniform distribution on $G$ is at most $(1/e)^{6 \ln n} \le
1/n^6$. In particular, for any vertex $v$ of $G$, $|\Pr[Y'_i = v ] - 1/n
|\le 1/n^6$.

Thus, for any $1<\ell\le n^3$ and any sequence $v_1,\dots,v_{\ell}$ of
vertices
$$
(1/n-/n^6)^{\ell}\le \Pr[Y'_1Y'_2\cdots Y'_{\ell}=v_1\cdots v_{\ell}]\le
(1/n+1/n^6)^{\ell}
$$
Hence,
$$
1/n^{\ell}\cdot (1-1/n^2) \le \Pr[Y'_1Y'_2\cdots Y'_{\ell}=v_1\cdots
v_{\ell}]\le 1/n^{\ell}\cdot(1 + 2/n^2).
$$

One can easily show (see the proof of Theorem~\ref{thm:t4}) that the probability
that a clique of size $n$ is not covered within $10n\ln n$ steps by a
random $1$-walk is at most $1/n^9$. By the above bound distribution of
$Y'_1Y'_2\cdots Y'_{\ell}$, for $1<\ell\le n^3$ is close to a distribution
of a random walk on a clique. Hence, unless $Y'_1,Y'_2,\dots,Y'_{10n\ln
n}$ does not hit all the vertices of $G$, we can bound the expected cover
time of $G$ by $(6 t_m \ln n) \cdot C^{k}(K_n) \cdot (1+2/n^2)$. If
$Y'_1,Y'_2,\dots,Y'_{10n\ln n}$ does not hit all the vertices of $G$ we
can bound the cover time of $G$ by the trivial bound $O(n^3)$. Since
$C^{k}(K_n) =  O(n \ln n/k)$ the claim follows.
%For expanders, the mixing time is $O(\log n)$ and in order to achieve the
%statistical distance $<1/n^6$ of the vertex distribution given by a random
%walk of length $\ell$ from the uniform distribution one only needs
%$\ell = O(\log n)$.
\end{proof}

\section{Logarithmic speed-up}
\label{s-line}

%MARK: modified start of transition paragraph

So far we have seen only cases where the speed-up in cover time achieved by multiple
random walks is considerable, i.e., at least linear. In this section we show that this is not always the
case and that the speed-up may be as low as logarithmic in $k$.
The cover time of a cycle $L_n$ on $n$ vertices is $\Theta(n^2)$.
We prove the following claim.
%\begin{thm}\label{thm:line}
\\ \\
\textbf{Theorem~\ref{thm:cycle}}
\emph{
For any integer $n$ and $k<\e^{n/4}$, the speed-up on the cycle with $n$ vertices
is $S^k(L_n) = \Theta(\log k)$.}
%\end{thm}
\\ \\
Hence for a cycle even a moderate speed-up of $\omega(\log n)$ requires
super-polynomially many walks, and to achieve speed-up of $n^\epsilon$ one
requires $2^{\Omega(n^\epsilon)}$ walks. The theorem follows from the following
two lemmas.

\begin{lemma}\label{l-24} Let $s>1$ and $k\ge 1$ be such that $C^{k}\le n^{2}/s$ for a cycle of length
$n$. Then $k\ge \e^{s/16}/8$.
\end{lemma}

\begin{proof}
Assume that $C^{k}\le n^{2}/s$ and we will prove that $k \ge
\e^{s/16}/8$. Pick an arbitrary vertex $v$ of the graph. Clearly, the cover time starting from the
vertex $v$ is $C^{k}_v \le n^{2}/s$. Let a random variable $T_v$
be the cover time of a random $k$-walk starting from $v$. By
Markov inequality, $\Pr[T_v \ge 2n^{2}/s]\le 1/2$. Hence,
with probability at least $1/2$ one of the $k$ walks reaches the
vertex $v_{n/2}$ that is at distance $n/2$ from $v$ in at most
$2n^{2}/s$ steps.
For a single walk, if it reaches $v_{n/2}$ starting from $v$ in time
at most $2n^{2}/s$, then there is $1\le t \le 2n^{2}/s$ so
that the number of its steps to the right until time $t$ differs from the number
of its steps to the left by at least $n/2$. Given that this happens,
with probability $1/2$ the number of steps to the right
will differ from the number of steps to the left by at least $n/2$
also at time $2n^{2}/s$. This is because after time $t$ we will
increase the difference with the same probability as that we will
decrease it since the probability of going to the left is the same as the probability
of going to the right.
By Chernoff bound, $\Pr[$the number of steps to the left and to the right of a walk
differs by at least $n/2$ at time $2n^{2}/s] \le 2e^{-{s \cdot n^2\over
16 n^{2}}} \le 2e^{-s/16}$. Hence, the probability that
a particular walk reaches the vertex $v_{n/2}$ during $2n^{2}/s$ steps
is at most $4e^{-s/16}$.

Thus, $\Pr[$there exists a walk that reaches $v_{n/2}$ in time
at most $2 n^{2}/s]\le 4k  \cdot
\e^{-s/16}$. Since this probability must be at least $1/2$ we
conclude that ${\e^{s/16}\over 8} \le k$.
\end{proof}

\begin{lemma}\label{l-cycleub}
Let $k$ be large enough and $n$ be an integer. If $k\le \e^{n/4}$ then $C^k \le 2n^2/\ln k$ for a cycle of length $n$.
\end{lemma}

To prove this lemma we need the following folklore statement (see the appendix for the proof).

\begin{proposition}\label{prp:folk}
Let $c\ge 2$ be a constant. For every even integer $n\ge 16c^2$,
$$
 \e^{-3c^2-4} \le \Pr[(c-1) \sqrt{n} \le X - n/2 \le c \sqrt{n}] \le \e^{-2(c-1)^2},
$$
where $X$ is a sum of $n$ independent $0$-$1$ random variables that are $1$ with probability $1/2$.
\end{proposition}

\begin{proofof}{Lemma~\ref{l-cycleub}}
To prove that $C^k \le  {2 n^2\over \ln k}$, let $c=\sqrt{\ln k}/2$ and $\ell=n^2/4(c-1)^2$.
If a single walk during a random $k$-walk of length $\ell$ on a cycle of length $n$ makes in total at least
$\ell/2+n/2$ steps to the right then it traversed around the whole cycle. Note, $n/2 = \sqrt{\ell}(c-1)$.
By the previous proposition, $\Pr[$a single walk makes at least $\ell/2+n/2$ steps to the right
during a random walk of length $\ell] \ge \e^{-3c^2-4}\ge 1/k$, for $k$ large enough. Hence, $k$ walks walking in parallel at random
for $\ell$ steps fail to cover the whole cycle of length $n$ with probability at most $(1-1/k)^k < 1/\e$.
Thus $C^k \le \sum_{i=0}^{\infty}{1\over \e^i} \ell = \e \ell / (\e-1) \le  2 n^2 /\ln k$, for $k$ large enough.
\end{proofof}

Lemma~\ref{l-24} also implies the following claim.

\begin{thm}
Let $G_{n,d}$ be a $d$-dimensional grid (torus) on $n^{1/d}\times n^{1/d}\times \cdots n^{1/d}$ vertices, $d\ge 2$.
For any $k$, $C^k(G_{n,d}) \ge \Omega(n^{2/d}  / \log k)$.
\end{thm}

\begin{proof}
We prove the claim for $d=2$. The other cases are analogous. Consider the random $k$-walk on a $\sqrt{n}\times \sqrt{n}$ grid (torus).
We can project the position of each of the $k$ walks to the $x$ axis. This will give a distribution identical to a $k$-walk
on a cycle of size $\sqrt{n}$ where in each step we make a step to the left with probability $1/4$, step to the right
with probability $1/4$ and with the remaining probability $1/2$ we stay at the current vertex. In order for a $k$-walk
to cover the whole grid, this projected walk must cover the whole cycle. Thus the expected cover time of the grid
must be lower-bounded by the expected cover time for a cycle of size $\sqrt{n}$ which is $\Omega(n/\log k)$ by Lemma~\ref{l-24}.
(Note the steps in which we stay at the same vertex can only increase the cover time.)
\end{proof}
\begin{cor}
For a $2$-dimensional grid $G_{n,2}$, $S^k(G_{n,2}) \le O(\log^2 n \log k)$.
\end{cor}
This corollary together with Theorem~\ref{s1-thit} implies Theorem~\ref{t-grid}.

%%========================== barbell ======================================================

\section{Exponential speed-up}
\label{s-barbell}

\infig{3}{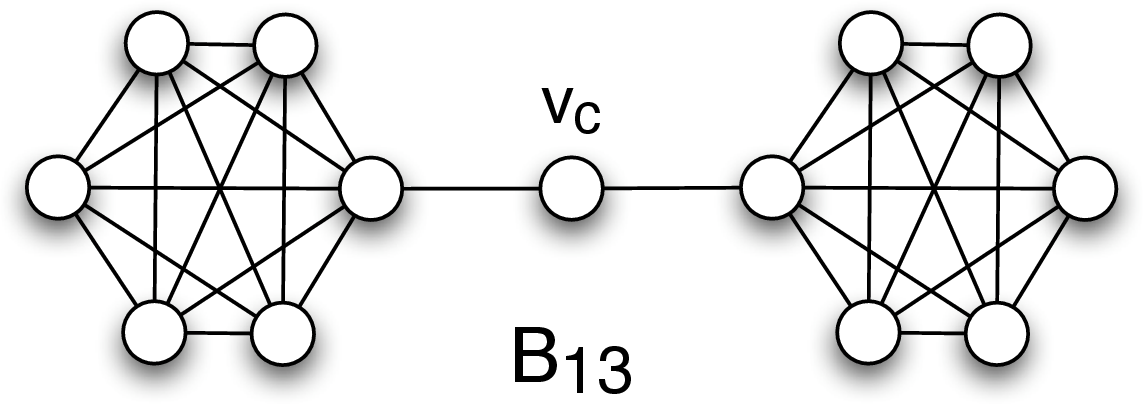}{Example barbell graph $B_{13}$, $v_c$ is the
center of the bar-ball}

\def\Pr{\mathop{\rm Pr}\nolimits}
%MARK: added next sentence to transition paragraph
On some graphs the speed-up can be exponential in $k$ for at least some choice of the starting point.
For an odd integer $n>1$, we define a barbell graph $B_n$ to be a
graph consisting of two cliques of size $(n-1)/2$ connected by a
path of length 2 (see Figure~\ref{fig:barball}). 
The vertex on that paths is called the {\em
center} of $B_n$ and the cliques are called {\em bells}. The
expected time to cover $B_n$ by a random walk is $\Theta(n^2)$ since
once the token is in one of the cliques it takes on average
$\Theta(n^2)$ steps to exit that clique. It can be shown that the
maximum cover time is attained by starting the random walk from the
center of $B_n$. We show the following theorem (see the appendix for the proof).

\begin{thm}\label{thm:t4}
Let $n>1$ be an odd integer, $v_{\rm c}$ be the center of $B_n$ and
$k=20 \ln n$. The expected cover time starting from $v_{\rm c}$
satisfies $C_{v_{\rm c}}^k = O(n)$.
\end{thm}

Hence, the speed-up in a cover time starting from a particular
vertex of a $k$-random walk
 compared to a random walk by a single token may be substantially larger than $k$.
In the case of $B_n$ the speed-up is $\Omega(n)$ for $O(\log n)$-walks for walks starting
at a particular vertex.

\section{Conclusions and Open Problems}
In this paper,
we have shown that many random walks can be faster than one,
sometimes much faster.
Our main result is that a linear speed-up is possible on
a large class of interesting graphs---including
complete graphs, expanders, grids, hypercubes, balanced trees, and
random graphs---in the sense that $k \leq \log n$ random walks can
cover an $n$-node graph $k$ times faster than a single random walk.
In the case of expanders, we obtain a linear speed-up even when $k$ is as large
as $n$.
Our technique is to relate the expected cover time for $k$ random walks
to the expected cover and hitting times for a single random walk;
and to observe that if there is a large gap between the single-walk
cover and hitting times,
then a linear speed-up is possible using multiple random walks.
%Chen: comment due to space
%The proof depends on an interesting generalization in
%Theorem~\ref{thm:generalMatthew}
%of Matthews' Theorem relating cover times and hitting times,
%generalizing his result from a single walk to multiple walks,
%and an elegant combinatorial argument in Lemma~\ref{l-composition}.
%MARK: shortened next sentence
%The hitting time is one interesting property of random walks,
%but the mixing time is another,
%and we are able to relate the cover time for $k$ random walks to the
%mixing time as well.
Using a different technique,
we were able to bound the $k$-walk cover time in terms
of the mixing time as well.

Open problems abound, despite of the progress reported here.
%MARK: shortened next sentence
%Of course, there are the standard questions concerning improving bounds.
There are the standard questions concerning improving bounds.
Is it possible that the speed-up is always at most $k$?
Our single counter example was that multiple random walks starting at the
center of the barbell achieved an exponential speed-up, but perhaps the
speed-up is limited to $k$ if we start at other nodes.
Is it possible that the speed-up is always at least $\log k$?
We have shown that the speed-up is $\log k$ on the ring,
%MARK: shortened next sentence
%and we conjecture that a $\log k$ speed-up is possible on any graph.
and we conjecture this is possible on any graph.

Another source of open problems is to consider more general classes of graphs.
Said in another way,
our approach has been to relate the $k$-walk cover time to the single-walk
hitting time and mixing time,
but is there another property of a graph that more crisply characterizes
the speed-up achieved by multiple random walks?

\bibliographystyle{acm}
\bibliography{spaa}
\appendix

\section{Proofs}

\begin{proofof}{Lemma \ref{l-composition}}
The proof is conceptually simple. We introduce here a little bit of notation to describe it formally.
For a sequence of vertices $\vec{c}=(c_0,c_1,\dots,c_t)$ and a random walk $X$ on $G$ starting from $c_0$,
$\vec{c} \sqsubseteq X$ denotes the event $\bigwedge_{i=0}^{t} X(i)=c_i$. For two sequences $\vec{c}=(c_0,\dots,c_t)$
and $\vec{d}=(d_0,\dots,d_{t'})$, where $c_t=d_0$ we denote by $\vec{c} \circ \vec{d} = (c_0,\dots,c_t,d_1,\dots, d_{t'})$.
It is straightforward to verify, if $X$ is a random walk starting from $c_0$ and
$Y$ is an independent random walk starting from $d_0$,
then $\Pr[ \vec{c} \sqsubseteq X \;\&\; \vec{d} \sqsubseteq Y] = \Pr[\vec{c} \circ \vec{d} \sqsubseteq X]$.
Last, for an integer $m \ge 1$ and a sequence $\vec{c}=(c_0,c_1,\dots,c_{km-1})$, $\vec{c}_{k,i}$ denotes
the subsequence $(c_{(i-1)m},\dots,c_{im-1})$ for $0 \le i \le k$.

WLOG $T_c$ is divisible by $k$.
Clearly, the probability that a random $k$-walk $(X_1,\dots,X_k)$ of length $T_c/k + \ell T_h$ on $G$
starting from vertices $u_1,\dots,u_k$ covers all of $G$ can be lower-bounded by
$$p = \Pr\left[\bigvee_{\vec{c}, \vec{h}_2,\dots,\vec{h}_k} \vec{c}_{k,1} \sqsubseteq X_1 \;\&\;
\vec{h}_2 \circ \vec{c}_{k,2} \sqsubseteq X_2 \;\&\; \cdots
\vec{h}_k \circ \vec{c}_{k,k} \sqsubseteq X_k \right],$$
where $\vec{c}$ is taken from the set of all sequences of vertices from $G$ corresponding to
walks of length $T_c$ on $G$ that start in $u_1$ and cover whole $G$, and $\vec{h}_i$ is taken
from the set of all sequences of vertices from $G$ corresponding to walks of length at most $\ell T_h$
that start in $u_i$ and hit $c_{(i-1)T_c/k}$ for the first time only at their end.
%
%
% For vertices $u$ and $v$ of $G$ and an integer $t>1$ denote by $W_u^t$ the set of all sequences
% of vertices from $G$ corresponding to walks of length $t$ on $G$ that start in $u$, and
% furthermore denote by $W_{u\rightarrow v}$ the set of all sequences of vertices from $G$ corresponding
% to walks that start in $u$, end in $v$ and visit $v$ for the first time only at the end of the walk.
%
% where $c \in W_{u_0}^{T_C}$ and covers $G$, and $h_i \in W_{u_i \rightarrow c_{T_c \cdot i/k}}$ of length
% at most $\ell T_h +1$.
%
%
It is easy to verify that all the events in the union are disjoint. Hence,
\begin{eqnarray*}
p &=& \sum_{\vec{c}, \vec{h}_2,\dots, \vec{h}_k}  \Pr\left[\vec{c}_{k,1} \sqsubseteq X_1 \;\&\;
\vec{h}_2 \circ \vec{c}_{k,2} \sqsubseteq X_2 \;\&\; \cdots
\vec{h}_k \circ \vec{c}_{k,k} \sqsubseteq X_k \right] \cr
  &=& \sum_{\vec{c}, \vec{h}_2,\dots, \vec{h}_k}  \Pr\left[\vec{c} \sqsubseteq X_1 \;\&\;
\vec{h}_2 \sqsubseteq X_2 \;\&\; \cdots
\vec{h}_k \sqsubseteq X_k \right] \cr
  &=& \sum_{\vec{c}, \vec{h}_2,\dots,\vec{h}_k}  \Pr[\vec{c} \sqsubseteq X_1 ] \cdot \Pr[\vec{h}_2 \sqsubseteq X_2] \cdots
\Pr[\vec{h}_k \sqsubseteq X_k ] \cr
  &=& \sum_{\vec{c}} \Pr[\vec{c} \sqsubseteq X_1] \cdot \sum_{\vec{h}_2} \Pr[\vec{h}_2 \sqsubseteq X_2] \cdots \sum_{\vec{h}_k} \Pr[\vec{h}_k \sqsubseteq X_k],
\end{eqnarray*}
where the third equality follows from the independence of the walks.
By our assumption $\sum_{\vec{c}} \Pr[\vec{c} \sqsubseteq X_1] \ge p_c$. Since
$(1-a)(1-b)\ge (1-a-b)$ for $a,b\le 1$, to conclude the lemma it
suffices to argue that $\sum_{\vec{h}_i} \Pr[\vec{h}_i \sqsubseteq X_i] \ge 1-
(1-p_h)^\ell$ for all $i$. Notice that $\sum_{\vec{h}_i} \Pr[\vec{h}_i \sqsubseteq X_i]=\Pr[$ a random
walk of length $\ell T_h$ starting from $u_i$ visits $c_{(i-1)T_c
/k}]$. Since a random walk of length $T_h$ fails to visit
$c_{(i-1)T_c /k}$ with probability at most $1-p_h$ regardless of its
starting vertex, a random walk of length $\ell T_h$ fails to visit
$c_{(i-1)T_c/k}$ with probability at most $(1-p_h)^\ell$. The lemma
follows.
\end{proofof}

\bigskip

\begin{proofof}{Proposition \ref{prp:folk}}
The upper bound follows from Chernoff bound.
The lower bound can be derived as follows. $\Pr[(c-1) \sqrt{n} \le X - n/2 \le c\sqrt{n}] = \sum_{k=(c-1)\sqrt{n}}^{c\sqrt{n}} \Pr[X-n/2=k]$.
For any $k$, $\Pr[X-n/2=k]= {n \choose n/2+k} / 2^{n}$.
We will compare ${n \choose n/2 + k}$ with the central binomial coefficient ${n\choose n/2}$.
\begin{eqnarray*}
{{n \choose n/2} \over {n \choose {n/2 + c\sqrt{n}}}} &=& \Pi_{j=1}^{n/2} {(n-j+1) \over j} \cdot \Pi_{j=1}^{n/2 + c\sqrt{n}} {j \over (n-j+1)}\cr
&=& \Pi_{j=n/2+1}^{n/2 + c\sqrt{n}} {j \over (n-j+1)}\cr
&=& \Pi_{j=1}^{c\sqrt{n}} {1+ {2\over n}j \over (1-{2\over n}(j+1))}.
\end{eqnarray*}

We upper-bound this ratio as follows:
\begin{eqnarray*}
\Pi_{j=1}^{c\sqrt{n}} ({1+ {2\over n}j }) &\le& \e^{{2 \over n}\sum_{j=1}^{c\sqrt{n}} j }\cr
&=& \e^{{2\over n}\cdot {{c\sqrt{n} (c\sqrt{n}+1)}\over 2}  }\cr
&\le& \e^{c^2 + 1}.
\end{eqnarray*}
%%%%
%% Here is a more elementary way to estimate the (somewhat weaker) lower bound:
%
% Now, $\Pi_{j=1}^{c\sqrt{n}} (1-{2\over n}(j+1)) =\Pi_{k=1}^{8c^2} \Pi_{j=1}^{\sqrt{n}/8c} (1-{2\over n}(j+1+(k-1)\sqrt{n}/8c))$,
% where for $1\le k \le 8c^2$
% \begin{eqnarray*}
% \Pi_{j=1}^{\sqrt{n}/8c} (1-{2\over n}(j+1+(k-1)\sqrt{n}/8c)) &\ge&  1-{2\over n}\sum_{j=1}^{\sqrt{n}/8c}(j+1+(k-1)\sqrt{n}/8c) \cr
% &\ge& 1- {2\over n}\sum_{j=1}^{\sqrt{n}/8c}(1+k \sqrt{n}/8c) \cr
% &\ge& 1- {2\over n} (k n/64c^2 +  \sqrt{n}/8c) \cr
% &\ge& 1/2.
% \end{eqnarray*}
Now, for $0\le x \le 1/2$, $\e^{-2x} \le 1-x$. Hence,
\begin{eqnarray*}
\Pi_{j=1}^{c\sqrt{n}} (1-{2\over n}(j+1)) &\ge& \e^{-{{4 \over n}\sum_{j=1}^{c\sqrt{n}} (j + 1) }} \cr
&\ge& \e^{-{4\over n}\cdot {{(c\sqrt{n}+1) (c\sqrt{n}+2)}\over 2}  }\cr
&\ge& \e^{-2c^2 - 2}.
\end{eqnarray*}

Thus $${{n \choose n/2} \over {n \choose {n/2 + c\sqrt{n}}}} \le \e^{3c^2+3}.$$
Using estimates on Stirling's formula ${n \choose n/2}\ge \sqrt{\frac{2}{\e \pi n}}\cdot 2^{n}$, we
conclude that
$$\sum_{k=(c-1)\sqrt{n}}^{c\sqrt{n}} {n \choose n/2 + k} \ge  {n \choose n/2} \sqrt{n} \e^{-3c^2-3} \ge \e^{-3c^2-4} \cdot 2^{n}.$$
The lemma follows.
\end{proofof}

\bigskip

\begin{proofof}{Theorem \ref{thm:t4}}
With high probability none of the following three events happens:
\begin{enumerate}
\item[$\mathcal{E}1$] In one of the bells there are less than $4\ln n$ tokens after the first step.
\item[$\mathcal{E}2$] During the first $10n$ steps of the random $k$-walk at least ${2} \ln n$
vertices return to the center.
\item[$\mathcal{E}3$] One of the bells is not covered within the first $10n$ steps.
\end{enumerate}

If none of the above events happens then each of the bells is
explored by at least $2\ln n$ tokens. Two disjoint cliques of size
$m=(n-1)/2$ are each covered by a random $2 \ln n$-walk in expected
time $2 C^{2\ln n}(K_m)$, by Lemma~\ref{lem:clique}. So if $C$
is the expected cover time of $B_n$ by a random $1$-walk then:
$$
C_{v_{\rm c}}^k \le                                       % \Pr[\overline{(1)\cup(2)\cup(3)}]
2 C^{2\ln n}(K_m) + \Pr[(1)]C + \Pr[\mathcal{E}2\cup\mathcal{E}3](10n+C).
$$

We need to estimate the probabilities of the above events. By
Chernoff bound, $$\Pr[\mathcal{E}1] \le 2e^{-(16 \ln n)^2/2 \cdot 20 \ln n} <
1/n^5$$ for $n$ large enough. A single token returns to the center of
$B_n$ within $10n$ steps with probability at most ${1\over n} + { 10
n\over m(m+1)} < {22\over m}$. The probability that at least $2\ln
n$ vertices return to the center is then $<2^{20 \ln n}\cdot
(22/m)^{2\ln n} < {1/n^5}$, for $n$ large enough. Finally, the
probability that a random $2\ln n$-walk does not cover a clique of
size $m$ in $10n$ steps is at most $m(1-{1\over m})^{20n \ln n}\le m
e^{-10 \ln n}<1/n^5$. Now since $C = O(n^2)$ and $C^{2\ln n}(K_m) = O(n)$,
we get $C_{v_{\rm c}}^k = O(n)$.
\end{proofof}

\bigskip

\begin{proofof}{Lemma \ref{lem:clique}}
In the lemma we restrict $k$ to be less than $n$ to avoid rounding problems and for simplicity
we also assume self loops in the clique.
We will prove this using a coupon collector argument. Let $C$ be the number of purchases needed to collect $n$ different coupons.
Consider the case where a fair mom decides to help her $k$ kids to collect the coupons. Each time she buys a cereal  and gets a coupon she
gives it to the next-in-turn son in a round-robin fashion (i.e. kid $i \text{ mod } k$ gets the coupon from step $i$). Clearly, in expectation, after
$C$ visits to the grocery store mom got all the different coupons. Note that each child had his own independent coupon collecting process, and each have the same number
of coupons (plus-minus one).
\end{proofof}

%\newpage
%\section*{To do list}

%\begin{enumerate}

%\item Add a Conclusion to a) summarize all our results and b) provide open problems.

%We don't have any other easy summary at the moment. Our results:
%linear speed-up for large class of graphs, relating speed-up and
%mixing time, logarithmic speed-up for cycle, exponential speed-up in
%special cases. Parallel version of Matthews' upper bound. Brag about
%our techniques (but don't overdo it, most of the proofs are
%technically easy).

%Open questions:
%\begin{enumerate}
%  \item is speed-up always at most $k$?
%  \item is speed-up always at least $\log k$?
%  \item what can you say about other classes of graphs?
%  \item is there some phenomenon (say ratio between the cover time and the mixing time)
%       which determines for which ranges we get what kind of speed-up
%       (What do I mean: what is the difference between line and expander?
%         but also: for $k>n^2$ there is (essentially) no more speed-up on expander.)
%  \item different starting points? - whole business.
%\end{enumerate}

%\item {\bf Proof read!} Run ispell.
%\item Proof read Section~\ref{s-kspeed}.
%\item Put some glue in sections on mixing time, barbells and cycle.
%\item Fix the table in intro.
%\end{enumerate}

\end{document}